\documentclass[11pt, oneside]{amsart}

\usepackage{tabularx, hyperref}
\usepackage{amssymb} \usepackage{amsfonts} \usepackage{amsmath}
\usepackage{amsthm} \usepackage{epsfig, subfig}
\usepackage{ amscd, amsxtra, latexsym}
\usepackage[all]{xy}
\usepackage{caption}
\usepackage{enumerate}
\usepackage{color}

\long\def\comment#1\endcomment{}
\usepackage{color}

\newtheorem{lemma}{Lemma}[section]
\newtheorem{thm}[lemma]{Theorem}
\newtheorem{prop}[lemma]{Proposition}
\newtheorem{cor}[lemma]{Corollary}

\theoremstyle{definition}
\newtheorem{defn}[lemma]{Definition}

\newtheorem{rem}[lemma]{Remark}
\newtheorem{conv}[lemma]{Convention}
\theoremstyle{definition}

\definecolor{darkgreen}{cmyk}{1,0,1,.2}

\newcommand{\g} {\ensuremath {\gamma}}

\newcommand{\N}{\ensuremath {\mathbb{N}}}

\newcommand{\calP} {\ensuremath {\mathcal{P}}}

\address{Mathematical Institute, 24-29 St Giles, Oxford OX1 3LB, United Kingdom}
\email{sisto@maths.ox.ac.uk}

\begin{document}

\title{Projections and relative hyperbolicity}
\author{Alessandro Sisto}

\maketitle
\begin{abstract}
We give an alternative definition of relative hyperbolicity based on properties of closest-point projections on peripheral subgroups. We also derive a distance formula for relatively hyperbolic groups, similar to the one for mapping class groups.
\end{abstract}

%\tableofcontents

\section*{Introduction}
The main aim of this paper is to introduce a new characterization of relatively hyperbolic groups in terms of projections on left cosets of peripheral subgroups. The properties we will consider are similar to those appeared in \cite{Be-asMCG,A-K-controut} and are used in \cite{Si-contr} in a more general setting. The characterization we will give is similar to the characterization of tree-graded spaces given in \cite{Si-univ}, the link being provided by asymptotic cones in view of results in \cite{DS-tg}. Our characterization only involves the geometry of the Cayley graph, alongside the ones given in \cite{DS-tg} and \cite{Dr-altrel}. Also, the statement deals with the more general setting of metric relative hyperbolicity (i.e. asymptotic tree-gradedness with the established terminology).
\par
We defer the exact statement to Section \ref{altdef}, see Definitions \ref{projsys}, \ref{altrfree} and Theorem \ref{projasymgrad:thm}.
\par
We will use projections also to provide an analogue for relatively hyperbolic groups of the distance formula for mapping class groups \cite{MM2}.
\par
Let $G$ be a relatively hyperbolic group and let $\calP$ be the collection of all left cosets of peripheral subgroups. For $P\in\calP$, let $\pi_P$ be a closest point projection map onto $P$. Denote by $\hat{G}$ the coned-off graph of $G$, that is to say the metric graph obtained from a Cayley graph of $G$ by adding an edge connecting each pair of (distinct) vertices contained in the same left coset of peripheral subgroup. 
Let $\big\{\big\{x\big\}\big\}_L$ denote $x$ if $x>L$, and $0$ otherwise.
We write $A\approx_{\lambda,\mu} B$ if $A/\lambda-\mu\leq B\leq \lambda A+\mu$.
\begin{thm}[Distance formula for relatively hyperbolic groups]
There exists $L_0$ so that for each $L\geq L_0$ there exist $\lambda,\mu$ so that the following holds. If $x,y\in G$ then
\begin{equation}
	d(x,y)\approx_{\lambda,\mu} \sum_{P\in\calP} \big\{\big\{d(\pi_P(x),\pi_P(y))\big\}\big\}_L+d_{\hat{G}}(x,y).
\end{equation}
\end{thm}

This formula will be used in \cite{MS-ecodim} to study quasi-isometric embeddings of relatively hyperbolic groups in products of trees.
It is useful for applications that projections admit alternative descriptions, see Lemma \ref{altproj}.
In subsection \ref{samplappl} we will give a sample application of the distance formula and show that a quasi-isometric embedding between relatively hyperbolic groups coarsely preserving left cosets of peripheral subgroups gives a quasi-isometric embedding of the corresponding coned-off graphs (the reader may wish to compare this result with \cite[Theorem 10.1]{Hr-relqconv}).

\subsection*{Acknowledgment}
The author would like to thank Cornelia Dru\c{t}u, Roberto Frigerio and John MacKay for helpful discussions and comments.

\section{Background on relatively hyperbolic groups}

\begin{defn}
A geodesic complete metric space $X$ is \emph{tree-graded} with respect to a collection $\calP$ of closed geodesic subsets of $X$ (called \emph{pieces}) if the following properties are satisfied:
\par
$(T_1)$ two different pieces intersect in at most one point,
\par
$(T_2)$ each geodesic simple triangle is contained in one piece.
%\par
%Moreover, if the pieces cover $\F$, then $\F$ is \emph{tree-graded} with respect to $\calP$.
\end{defn}

%\begin{rem}
%We do not consider trivial triangles to be simple, as it is done in~\cite{DS1}, where the pieces of a tree-graded space are required to cover it. The results below, however, do not depend on that.
%\end{rem}
Tree-graded spaces can be characterized in terms of closest-point projections on the pieces.
Let us denote by $X$ a complete geodesic metric space and by $\calP$ a collection of subsets of $X$. Consider the following properties.

\begin{defn}
A family of maps $\Pi=\{\pi_P:X\to P\}_{P\in \calP}$ will be called \emph{projection system for} $\calP$ if, for each $P\in\calP$,
\par
$(P1)$ for each $r\in P$, $z\in X$, $d(r,z)=d(r,\pi_P(z))+d(\pi_P(z),z)$,
\par
$(P2)$ $\pi_P$ is locally constant outside $P$,
\par
$(P3)$ for each $Q\in\calP$ with $P\neq Q$, we have that $\pi_P(Q)$ is a point.
\end{defn}

\begin{defn}
A geodesic is $\calP-$\emph{transverse} if it intersects each $P\in\calP$ in at most one point. A geodesic triangle in $X$ is $\calP-$\emph{transverse} if each side is $\calP-$transverse.
\par
$\calP$ is \emph{transverse-free} if each $\calP-$transverse geodesic triangle is a tripod.
\end{defn}

\begin{thm}\cite{Si-univ}
\label{projgrad:thm}
Let $X$ be a complete geodesic metric space and let $\calP$ a collection of subsets of $X$. Then $X$ is tree-graded with respect to $\calP$ if and only if $\calP$ is transverse-free and there exists a projection system for $\calP$.
%If $\calP$ covers $X$ then the same is true removing ``with gaps''.
\end{thm}

The following properties have also been considered in \cite{Si-univ}. Properties $(P_1)$ and $(P_2)$ are equivalent to $(P'_1)$ and $(P'_2)$.

\begin{lemma}
Properties $(P1)$ and $(P2)$ can be substituted by:
\par
$(P'1)$ for each $P\in\calP$ and $x\in P$, $\pi_P(x)=x$,
\par
$(P'2)$ for each $P\in\calP$ and for each $z_1,z_2\in X$ such that $\pi_P(z_1)\neq\pi_P(z_2)$,
$$d(z_1,z_2)=d(z_1,\pi_P(z_1))+d(\pi_P(z_1),\pi_P(z_2))+d(\pi_P(z_2),z_2).$$
\end{lemma}

The reader unfamiliar with asymptotic cones is referred to \cite{Dr-ascones}.
\begin{conv}
 Throughout the paper we fix a non-principal ultrafilter $\mu$ on $\N$. We will denote ultralimits by $\mu-\lim$ and the asymptotic cone of $X$ with respect to (the ultrafilter $\mu$,) the sequence of basepoints $(p_n)$ and the sequence of scaling factor $(r_n)$ by $C(X,(p_n),(r_n))$.
\end{conv}

\begin{defn}\cite{DS-tg}
 The geodesic metric space $X$ is \emph{asymptotically tree-graded} with respect to the collection of subsets $\calP$ if all its asymptotic cones, with respect to the fixed ultrafilter, are tree-graded with respect to the collection of the ultralimits of elements of $\calP$.
\end{defn}

\begin{defn}
 The finitely generated group $G$ is \emph{hyperbolic relative} to its subgroups $H_1,\dots,H_n$, called \emph{peripheral subgroups}, if its Cayley graphs are asymptotically tree-graded with respect to the collection of all left cosets of the $H_i$'s.
\end{defn}

Let $X$ be asymptotically tree-graded with respect to $\calP$. We report below some useful lemmas from \cite{DS-tg} that will be used later.
%Let us start with property $(\alpha_1)$ and (a slight modification of) property $(\alpha_2)$ as in Theorem 4.1 in~\cite{DS1}.

When $A$ is a subset of the metric space $X$, the notation $N_d(A)$ will denote the closed neighborhood of radius $d$ around $A$, i.e. $N_d(A)=\{x\in X| d(x,A)\leq d\}$.

\begin{lemma}\cite[Theorem 4.1$-(\alpha_2)$]{DS-tg}
\label{M:lem}
If $\gamma$ is a geodesic connecting $x$ to $y$, and $d(x,P),d(y,P)\leq d(x,y)/3$ for some $P\in\calP$, then $\gamma\cap N_M(P)\neq\emptyset$.
%Also, there exists $\sigma$ such that, for each $C\geq 1$, $M$ can be chosen to be $\sigma C$.
\end{lemma}

\begin{lemma}{\cite[Lemma 4.7]{DS-tg}}\label{coarseinters:lem}
For each $H\geq0$ there exists $B$ such that $diam(N_H(P)\cap N_H(Q))\leq B$ for each $P, Q\in\calP$ with $P\neq Q$.
\end{lemma}

We will also need that each $P\in\calP$ is \emph{quasi-convex}, in the following sense.

\begin{lemma}{\cite[Lemma 4.3]{DS-tg}}\label{qconv:lem}
There exists $t$ such that for each $L\geq 1$ each geodesic connecting $x,y\in N_L(P)$ is contained in $N_{tL}(P)$.
\end{lemma}

If $G$ is hyperbolic relative to $H_1,\dots,H_n$, its \emph{coned-off} graph, denoted $\hat{G}$, is obtained from a Cayley graph of $G$ by adding edges connecting vertices lying in the same left coset of peripheral subgroup.
\par
By \cite{Fa}, $\hat{G}$ is hyperbolic and the following property holds.

\begin{prop}[BCP property]
 Let $\alpha,\beta$ be geodesics in $\hat{G}$, for $G$ relatively hyperbolic, and let $\calP$ be the collection of all left cosets of peripheral subgroups of $G$. There exists $c$ with the following property.
\begin{enumerate}
 \item If $\alpha$ contains an edge connecting vertices of some $P\in\calP$ but $\beta$ does not, then such vertices are at distance at most $c$ in $G$.
 \item If $\alpha$ and $\beta$ contain edges $[p_\alpha,q_\alpha], [p_\beta,q_\beta]$ (respectively) connecting vertices of some $P\in\calP$, then $d_G(p_\alpha,p_\beta),d_G(q_\alpha,q_\beta)\leq c$.
\end{enumerate}
\end{prop}

\subsection{Geodesics and projections}
\begin{conv}
 In this subsection $X$ is an asymptotically tree-graded space with respect to a collection of subsets $\calP$. Sometimes we will restrict to $X$ a Cayley graph of a relatively hyperbolic group, and in that case $\calP$ will always be the collection of left cosets of peripheral subgroups.
\end{conv}

%To simplify statements and proofs, throughout the subsection we fix constants $k,c\geq 0$ and by almost geodesic we mean continuous $(k,c)-$quasi-geodesic. The constants we will find depend on $k,c$.
%The statements in this subsection which do not involve "first points" with some property hold also for quasi-geodesics up to changing the constants, see Remark~\ref{contqgeod:rem}
%\par
%Let us now prove some useful properties of almost geodesics in $X$.
The following definition is taken from~\cite{DS-tg} (Definition 4.9).

\begin{defn}
If $x\in X$ and $P\in\calP$ define the almost projection $\pi_P(x)$ to be the subset of $P$ of points whose distance from $x$ is less than $d(x,P)+1$.
\end{defn}

The following lemma gives two alternative characterizations of the maps $\pi_P$.
\begin{lemma}
\
 \begin{enumerate}
  \item If $\alpha$ is a continuous $(K,C)-$quasi-geodesic connecting $x$ to $P\in\calP$ then for each $D\geq D_0=D_0(K,C)$ there exists $M$ so that the first point in $\alpha\cap N_D(P)$ is at distance at most $M$ from $\pi_P(x)$.
 \item {[Bounded Geodesic Image]} If $X$ is the Cayley graph of $G$, there exists $M$ so that if $\hat{\gamma}$ is a geodesic in $\hat{G}$ connecting $x\in G$ to $P\in\calP$ then the first point in $\hat{\gamma}\cap P$ is at distance at most $M$ from $\pi_P(x)$.
 \end{enumerate}
\label{altproj}
\end{lemma}

\begin{proof}
 $(1)$ 
 The \emph{saturation} of a geodesic is the union of the geodesic and all
 $P\in\calP$ whose $\mu-$neighborhood intersects the geodesic (for some appropriately chosen $\mu$). 
 By \cite[Lemma 4.25]{DS-tg} there exists $R=R(K,C)$ so that if $\gamma$ is a geodesic 
 and the $(K,C)-$quasi-geodesic $\alpha$ 
 connects points in the saturation $Sat(\gamma)$ of $\gamma$, then $\alpha$ is contained in the $R-$neighborhood of
 $Sat(\gamma)$. 
 
 We can apply this when $\alpha$ is as in our statement and $\gamma$ is a geodesic from $x$ to $\pi_P(x)$.
 Let $D\geq \mu,R$ and let $p$ be the first point in $\alpha\cap N_D(P)$. There are two cases to consider. 
 If $p\in N_R(\gamma)$ then we are done as $diam(\gamma\cap N_{D+R}(P))\leq D+R$ and
 $p,\pi_P(x)\in N_R(\gamma)\cap N_{D}(P)$. 
 Otherwise there exists $P'\neq P$ so that $P'\subseteq Sat(\gamma)$ and $p\in N_R(P')$. 
 By \cite[Lemma 4.28]{DS-tg} there exists $B=B(D)$ so that $N_D(P)\cap N_R(P')\subseteq N_B(\gamma)$. 
 As noticed earlier $diam(\gamma\cap N_{B+D}(P))\leq B+D$ and $\pi_P(x),p\in N_B(\gamma)\cap N_{D}(P)$, so we are done.
\par
$2)$ Let $\hat{\gamma}_0$ be a geodesic in $\hat{G}$ connecting $x$ to $\pi_P(x)$ and denote by $p$ the first point in $\hat{\gamma}\cap P$, and let $\hat{\gamma}_1$ be any geodesic from $x$ to $P$ intersecting $P$ only in its endpoint $q$. By adding an edge to $\hat{\gamma}_1$ connecting $q$ to $\pi_P(x)$ we are in a situation where we can apply the BCP property to get a uniform bound on $d(p,q)$. So, it is enough to prove the statement for $\hat{\gamma}=\hat{\gamma}_0$. By \cite[Lemma 8.8]{Hr-relqconv}, we can bound by some constant, say $B$, the distance from $p$ to a geodesic $\gamma$ in $G$ from $x$ to $\pi_P(x)$. As in the first part, we have $p,\pi_P(x)\in N_B(\gamma)\cap N_D(P)$, a set whose diameter can be bounded by $B+D$.
\end{proof}

The \emph{lift} of a geodesic in $\hat{G}$ is a path in $G$ obtained by substituting edges labeled by an element
of some $H_i$ and possibly the endpoints with a geodesic in the corresponding left coset. The following is a consequence of \cite[Lemma 8.8]{Hr-relqconv} (or of the distance formula and the second part of Lemma \ref{altproj}, but \cite[Lemma 8.8]{Hr-relqconv} is used in the proof).

\begin{prop}[{Hierarchy paths for relatively hyperbolic groups}]\label{thm-lifts-qgeod}
	There exist $\lambda, \mu$ so that if $\alpha$ is a geodesic in $\hat{G}$ then its lifts are $(\lambda,\mu)-$quasi-geodesics.
\end{prop}

\begin{lemma}
\label{proj2:lem}
 There exists $L$ so that if $d(\pi_P(x),\pi_P(y))\geq L$ for some $P\in\calP$ then
\begin{enumerate}
 \item all $(K,C)-$quasi-geodesics connecting $x$ to $y$ intersect $B_R(\pi_P(x))$ and $B_R(\pi_P(y))$, where $R=R(K,C)$,
 \item all geodesics in $\hat{G}$ connecting $x$ to $y$ contain an edge in $P$, when $X$ is a Cayley graph of $G$.
\end{enumerate}
\end{lemma}

\begin{proof}
 In view of Lemma \ref{altproj}$-(1)$, in order to show $1)$ we just have to show that any quasi-geodesic $\alpha$ as in the statement intersects a neighborhood of $P$ of uniformly bounded radius. We can suppose that $\gamma$ is continuous. Let $p$ be a point on $\alpha$ minimizing the distance from $P$, and let $\gamma$ be a geodesic from $p$ to $P$ of length $d(p,P)$. The point $p$ splits $\alpha$ in two halves $\alpha_1,\alpha_2$, and it is easy to show that the concatenation $\beta_i$ of $\alpha_i$ and $\gamma$ is a quasi-geodesic with uniformly bounded constants:
\begin{lemma}
 Let $\delta_0$ be a geodesic connecting $q$ to $p$ and let $\delta_1$ be a $(K,C)-$quasi-geodesic starting at $p$. Suppose that $d(q,p)=d(q,\delta_1)$. Then the concatenation $\delta$ of $\delta_0$ and $\delta_1$ is a $(K',C')-$quasi-geodesic, for $K',C'$ depending on $K,C$ only.
\end{lemma}

\begin{proof}
It is clear that the said concatenation is coarsely lipschitz.
Let $I=I_0\cup I_1$ be the domain of $\delta$, where $I_0,I_1$ are (translations of) the domains of $\delta_0,\delta_1$. We will denote by $t$ the intersection of $I_0$ and $I_1$ so that $\delta(t)=\delta_0(t)=\delta_1(t)=p$. Let $t_0,t_1\in I$ and set $x_i=\delta(t_i)$. We can assume $t_i\in I_i$, the other cases being either symmetric or trivial. Suppose first $d(x_0,p)=|t-t_0|\leq |t-t_1|/(2K) -C/2$. In this case $d(x_0,p)\leq d(x_1,p)/2$ so that $d(p,x_1)\leq d(p,x_0)+d(x_0,x_1)\leq d(p,x_1)/2+d(x_0,x_1)$ and hence $d(p,x_1)\leq 2d(x_0,x_1)$. Then
$$|t_0-t_1|=|t_0-t|+|t-t_1|\leq 3 d(p,x_1)/2\leq 3d(x_0,x_1).$$
On the other hand, if $|t-t_0|\geq |t-t_1|/2K -C/2$ then
$$|t_0-t_1|\leq (2K+1)d(x_0,p)+KC\leq (2K+1)d(x_0,x_1)+KC,$$
%\geq d(p,x_1)/(2K^2)-C/2-C/(2K^2)
as $d(x_0,p)\leq d(x_0,x_1)$.
\end{proof}

Again by Lemma \ref{altproj}$-(1)$ we can uniformly bound the distance between the projections of $x$ and $y$ on $P$ if $d(p,P)> D_0=D_0(K,C)$, so that $\beta_i\cap N_{D_0}(P)=\gamma\cap N_{D_0}(P)$, so that for $L$ large enough we must have $\alpha\cap N_{D_0}(P)\neq\emptyset$ as required.
\par
Let $\hat{\gamma}$ be a geodesic in $\hat{G}$. Part $1)$ applies in particular to lifts $\hat{\gamma}$, so that the conclusion follows applying the BCP property to a sub-geodesic of $\hat{\gamma}$ connecting points close to $\pi_P(x)$ to $\pi_P(y)$ and the geodesic in $\hat{G}$ consisting of a single edge connecting $\pi_P(x)$ to $\pi_P(y)$.
\end{proof}

%\begin{lemma}\label{closeproj:lem}
%For each $p,q\in X$, $A\in\calP$, $d(\pi_A(p),\pi_A(q))\leq 2d(p,q)+2R+1$.

%\end{lemma}

% \begin{proof}
% Set $C=d(p,q)$. Fix $\hat{p}\in\pi_A(p)$ and $\hat{q}\in\pi_A(q)$ and a geodesic $[p,\hat{q}]$. By Lemma~\ref{proj1:lem} and Remark~\ref{boundeddiam:rem}, there exists $x\in [p,\hat{q}]$ such that $d(\hat{p},x)\leq 2R$.
% %We have $$d(p,x)\geq d(p,\hat{p})-2R,$$
% %\geq d(q,\hat{q})-d(\hat{p},q)-d(q,p)-2R.$$
% We have
% $$d(p,\hat{q})\leq d(p,q)+d(q,\hat{q})\leq C+d(q,P)+1\leq d(q,x)+d(x,\hat{p})+C+1\leq d(q,x)+C+R+1,$$
% hence $d(q,x)\geq d(p,\hat{q})-C-R-1$ and $d(p,x)\geq d(p,\hat{q})-2C-R-1$.
% As $d(x,\hat{q})=d(p,\hat{q})-d(p,x)$, we get 
% $$d(x,\hat{q})\leq 2C+R+1.$$
% So, $d(\hat{p},\hat{q})\leq d(\hat{p},x)+d(x,\hat{q})\leq 2C+2R+1$.
% \end{proof}

\section{Alternative definition of relative hyperbolicity}
\label{altdef}
In this section we state the analogue of the alternative definition of tree-graded spaces that can be found in \cite{Si-univ}. Throughout the section let $X$ be a geodesic metric space and let $\calP$ be a collection of subsets of $X$.
\par
We will need the coarse versions of the definitions of projection system and being transverse-free, as defined in \cite{Si-univ}.

\begin{defn}
\label{projsys}
A family of maps $\Pi=\{\pi_P:X\to P\}_{P\in \calP}$ will be called \emph{almost-projection system for} $\calP$ if there exist $C\geq 0$ such that, for each $P\in\calP$,
\par
$(AP1)$ for each $x\in X$, $p\in P$, $d(x,p)\geq d(x,\pi_P(x))+d(\pi_P(x),p)-C$,
\par
$(AP2)$ for each $x\in X$ with $d(x,P)=d$, $diam\left(\pi_P(B_d(x))\right)\leq C$,
\par
$(AP3)$ for each $P\neq Q\in\calP$, $diam(\pi_P(Q))\leq C$.
\end{defn}

\begin{rem}\label{ap'1:rem}
For each $x\in X$ and $P\in\calP$, by $(AP1)$ with $p=\pi_P(x)$ we have $d(x,\pi_P(x))\leq d(x,P)+C$.
\end{rem}

\subsection{Technical lemmas}
First of all, let us prove some basic lemmas.
One of the aims will be to prove that properties $(AP1)$ and $(AP2)$ are equivalent to coarse versions of properties $(P'1)$ and $(P'2)$ that will be formulated later.
\par
Consider an almost-projection system for $\calP$ and let $C$ be large enough so that $(AP1)$ and $(AP2)$ hold. Let us start by proving that projections are coarsely contractive, in 2 different senses.
%The following lemma will be be very important in the rest of the paper.

\begin{lemma}\label{coarsedistdecr:lem}
\
\begin{enumerate}
\item
Consider some $k\geq 1$ and a path $\g$ connecting $x$ to $y$ such that $d(x,P)\geq kC$ for each $x\in\g$. Then $d(\pi_P(x),\pi_P(y))\leq l(\g)/k+C$.
\item
$d(\pi_P(x),\pi_P(y))\leq d(x,y)+ 6C$.
\end{enumerate}
\end{lemma}

\begin{proof}
$(1):$ Consider a partition of $\gamma$ in subpaths $\gamma_i=[x_i,y_i]$ of length $kC$ and one subpath $\gamma'=[x',y']$ of length at most $kC$. By property $(AP2)$ we have $d(\pi_P(x_i),\pi_P(y_i))\leq C=d(x_i,y_i)/k$ and $d(\pi_P(x'),\pi_P(y'))\leq C$, so
$$d(\pi_P(x),\pi_P(y))\leq\sum d(\pi_P(x_i),\pi_P(y_i))+d(\pi_P(x'),\pi_P(y'))\leq$$
$$ \sum d(x_i,y_i)/k+d(x',y')/k+C\leq l(\gamma)/k+C.$$
\par
$(2):$ Consider a geodesic $\g$ connecting $x$ to $y$. If $\g\cap N_C(P)=\emptyset$ we can apply the first point. Otherwise, let $\gamma'=[x,x']$ (resp. $\gamma''=[y',y]$) be a (possibly trivial) subgeodesic such that $\g'\cap N_C(P)=x'$ (resp. $\g''\cap N_C(P)=y'$). Applying the previous point to $\g'$ and $\g''$ and Remark~\ref{ap'1:rem} we get
$$d(\pi_P(x),\pi_P(y))\leq $$
$$d(\pi_P(x),\pi_P(x'))+d(\pi_P(x'),x')+d(x',y')+d(y',\pi_P(y'))+d(\pi_P(y'),\pi_P(y))\leq $$
$$ (d(x,x')+C)+2C+d(x',y')+2C+(d(y',y)+C)=d(x,y)+6C,$$
as required.
\end{proof}

%As an application we have the following.

%\begin{cor}\label{Pisqconv:cor}
%Each $P\in\calP$ is quasi-convex.
%\end{cor}

%\begin{proof}
%Consider any geodesic $\delta$ connecting points in $N_{kC}(P)$, for some $k\geq 1$. Consider a subgeodesic $\gamma$ with endpoints $x,y$ such that $\gamma\cap N_{kC}(P)=\{x,y\}$. Then
%$$l(\g)\leq (2k+2)C+d(\pi_P(x),\pi_P(y))\leq l(\g)/k+(2k+3)C,$$
%so $l(\g)\leq (4k+6)C$ (this is enough). This implies $\delta\subseteq N_{(3k+3)C}(P)$.
%\end{proof}

%Now let us prove that the projection is coarsely constant along any geodesic connecting a point to its projection on some $P$.

%\begin{lemma}
%\label{geodboundedproj:lem}
%Let $\gamma$ be a geodesic connecting $x$ to $\pi_P(x)$, for some $P\in\calP$. Then $diam(\pi_P(\gamma))\leq 6C$.
%\end{lemma}

%\begin{proof}
%Consider the initial subgeodesic $\gamma'$ of $\gamma$ of length $d(x,P)$. By $(AP2)$ we have $diam(\pi_P(\gamma'))\leq C$. Let $y$ be the ending point of $\gamma'$ and let $\gamma''$ be the final subgeodesic of $\gamma$ starting at $y$. By Remark~\ref{ap'1:rem}, we have $l(\gamma'')\leq C$ and so $d(z,P)\leq C$ for each $z\in \gamma''$. In particular $d(z,\pi_P(z))\leq 2C$ for each $z\in\gamma''$, by Remark~\ref{ap'1:rem}. So, $diam(\pi_P(\gamma''))\leq 5C$, hence the thesis as $\pi_P(\gamma')\cap\pi_P(\gamma'')\neq\emptyset$.
%\end{proof}

%The following two lemmas provide coarse versions of Lemma~\ref{propproj:lem}$-(1)$.

\begin{lemma}\label{coarsepropproj1:lem}
For each $r$ and $c\geq 0$ we have that each $(1,c)-$quasi-geodesic $\gamma$ from $x\in X$ to $y\in N_r(P)$, for some $P\in\calP$, intersects $B_\rho(\pi_P(x))$, where $\rho=2r+6C+5c$. Moreover, any point $y'$ on $\gamma$ such that $d(x,P)-2c\leq d(x,y')\leq d(x,P)$ belongs to $B_\rho(\pi_P(x))$.
\end{lemma}

\begin{proof}
Note that $y'$ as in the statement exists if and only if $d(x,y)\geq d(x,P)-2c$. Suppose $d(x,y)<d(x,P)-2c$. In this case $d(\pi_P(x),\pi_P(y))\leq C$ by $(AP2)$, so $d(y,\pi_P(x))\leq r+2C$ (we used Remark~\ref{ap'1:rem}).
\par
Let us now consider the other case. Let $y'\in\gamma$ be such that $d(x,P)-2c\leq d(x,y')\leq d(x,P)$ and let $\gamma'$ be the sub-quasi-geodesic of $\gamma$ from $x$ to $y'$. As $d(y,\pi_P(y))\leq r+C$ and $d(\pi_P(y'),\pi_P(x))\leq C$, we have, using $(AP1)$ in the second inequality,
$$d(y',y)\geq d(y',\pi_P(y))-r-C\geq d(y',\pi_P(y'))+d(\pi_P(y'),\pi_P(y))-r-2C\geq $$
$$d(y',\pi_P(x))+d(\pi_P(x),\pi_P(y))-r-4C.$$
Also,
$$d(x,y)\leq d(x,\pi_P(x))+d(\pi_P(x),\pi_P(y))+r+C.$$

As $d(x,y)\geq d(x,y')+d(y',y)-3c$ (as these points lie on a $(1,c)-$quasi-geodesic) and $d(x,y')\geq d(x,P)-2c$, we obtain
$$[d(y',\pi_P(x))+d(\pi_P(x),\pi_P(y))-r-4C]+d(x,P)\leq $$
$$d(y',y)+d(y',x)+2c\leq d(x,y)+5c\leq $$
$$d(x,\pi_P(x))+d(\pi_P(x),\pi_P(y))+r+C+5c\leq d(x,P)+d(\pi_P(x),\pi_P(y))+r+2C+5c.$$
Therefore,
$$d(y',\pi_P(x))\leq 2r+6C+5c.$$
\end{proof}

The following can be thought as another coarse version of property $(P1)$.

\begin{lemma}\label{coarsepropproj2:lem}
Consider a geodesic $\gamma$ starting from $x$ and some $P\in\calP$ such that $\gamma\cap N_r(P)\neq\emptyset$, for some $r\geq 2C$. Let $y$ be the first point on $\gamma$ in $N_r(P)$. Then $d(y,\pi_P(x))\leq 8r+22C$.
\end{lemma}

\begin{proof}
If $d(x,y)\leq d(x,P)$, we have $d(\pi_P(x),\pi_P(y))\leq C$ by $(AP1)$, so $d(y,\pi_P(x))\leq r+2C$ (we used Remark~\ref{ap'1:rem}).
Suppose that this is not the case and let $y'$ be as in the previous lemma. Consider a geodesic $\gamma'=[y,y']$.
\par
By $d(y,\pi_P(y))\leq r+C$, $d(y',\pi_P(y'))\leq 2r+7C$ (because of Remark~\ref{ap'1:rem}), Lemma~\ref{coarsedistdecr:lem}$-(1)$ with $k=2$ (recall that $r\geq 2C$ and notice that $\gamma'\cap N_r(P)=\{y\}$), we have
$$d(y,y')\leq d(y,\pi_P(y))+d(\pi_P(y),\pi_P(y'))+d(\pi_P(y'),y')\leq 3r+8C+d(y,y')/2.$$
So, $d(y,y')\leq 6r+16C$ and $d(y,\pi_P(x))\leq d(y,y')+d(y',\pi_P(x))\leq 8r+22C$.
\end{proof}

\begin{cor}\label{boundedproj:cor}
Consider a geodesic $\g$ from $x$ to $y$ and some $P\in\calP$ such that $\gamma\cap N_r(P)=\{y\}$, for some $r\geq 2C$. Then $l(\g)\leq d(x,P)+8r+23C$ and $\pi_P(\g)\subseteq B_{8r+30C}(\pi_P(x))$.
\end{cor}

\begin{proof}
Using the previous lemma, $l(\g)=d(x,y)\leq d(x,\pi_P(x))+d(\pi_P(x),y)\leq d(x,P)+C+(8r+22C)$. The second part is an easy consequence of this fact, using $(AP2)$ and Lemma~\ref{coarsedistdecr:lem}$-(2)$.
\end{proof}

\begin{cor}\label{projneighb:cor}
Let $\g$ be a geodesic from $x_1$ to $x_2$. Then $diam(\g\cap N_r(P))\leq d(\pi_P(x_1),\pi_P(x_2))+18r+62 C$ for each $r\geq 2C$ and $P\in\calP$.
\end{cor}

\begin{proof}
Let $x'_1,x'_2$ be the first and last point in $\g\cap N_r(P)$. By Corollary~\ref{boundedproj:cor}, we have $d(\pi_P(x_i),\pi_P(x'_i))\leq 8r+30C$. So, $d(\pi_P(x_1),\pi_P(x_2))\geq d(x'_1,x'_2)-2(8r+30C)-2(r+C)=d(x'_1,x'_2)-18r-62C$. As $d(x'_1,x'_2)=diam(\g\cap N_r(P))$, this is what we wanted.
\end{proof}

We will consider the following coarse analogs of properties $(P'1)$ and $(P'2)$.

\par
$(AP'1)$ There exists $C\geq 0$ such that for each $x\in X$, $d(x,\pi_P(x))\leq d(x,P)+C$.
\par
$(AP'2)$ There exists $C\geq 0$ with the property that for each $x_1,x_2\in X$ such that $d(\pi_P(x_1),\pi_P(x_2))\geq C$, we have
$$d(x_1,x_2)\geq d(x_1,\pi_P(x_1))+d(\pi_P(x_1),\pi_P(x_2))+d(\pi_P(x_2),x_2)-C.$$
%\par
%Now, let us prove the equivalent of Lemma~\ref{alternproj:lem} for almost-projections.

\begin{lemma}\label{alternasymproj:lem}
$(AP1)+(AP2)\iff (AP'1)+(AP'2)$.

\end{lemma}

\begin{defn}
We will say that $C$ is a projection constant if the properties $(AP1),(AP2),(AP'1),(AP'2)$ hold with constant $C$.
\end{defn}

\begin{proof}
$\Leftarrow:$ Fix $C$ large enough so that $(AP'1),(AP'2)$ hold. Property $(AP1)$ is not trivial only if $d(\pi_P(x),x)$ is large, and in this case it follows from $(AP'2)$ setting $x_1=x$ and $x_2=\pi_P(x)=p$ and keeping into account $d(\pi_P(p),p)\leq C$. Let us show property $(AP2)$. Note that $d(\pi_P(x),\pi_P(x'))> C$ implies $d(x,x')> d(x,P)-2C$. We want to exploit this fact. Set $d=d(x,P)$. Note that if $x'\in B(x,d)$, then there exists $x''\in B_{d-2C}$ such that $d(x',x'')\leq 2C$ and one of of the following 2 cases holds:
\begin{itemize}
\item
$x'\in N_{6C}(P)$, or
\item
$d(x'',P)\geq 4C$.
\end{itemize}

In the first case either $d(\pi_P(x'),\pi_P(x''))<C$ or
$$d(x',\pi_P(x'))+d(\pi_P(x'),\pi_P(x''))+d(\pi_P(x''),x'')-C\leq d(x',x'')\leq 2C,$$
and so $d(\pi_P(x'),\pi_P(x''))\leq 3C$. In the second case $d(x',x'')\leq d(x',P)-2C$, and so $d(\pi_P(x'),\pi_P(x''))\leq C$.
\par
These considerations yield $diam\left(\pi_P(B_d(x))\right)\leq 4C$.
\par
$\Rightarrow:$ We already remarked that $(AP'1)$ holds. Let $C>0$ be large enough so that $(AP1)$ and $(AP2)$ hold.
We will prove the following, which implies $(AP'2)$ setting $c=0$ and which will be useful later.

\begin{lemma}\label{morethanap'2:lem}
If $d(\pi_P(x_1),\pi_P(x_2))\geq 8C+8c+1$, for some $c\geq 0$ and $P\in\calP$, then any $(1,c)-$quasi-geodesic $\gamma$ from $x_1$ to $x_2$ intersects $N_{2C}(P)$ and $B_{10C+5c}(\pi_P(x_i))$.
\end{lemma}

\begin{proof}
Once we show that $\gamma\cap N_{2C}(P)\neq\emptyset$, we can apply Lemma~\ref{coarsepropproj1:lem} to obtain $B_{10C+5c}(\pi_P(x_i))\cap \gamma\neq\emptyset$
\par
Set $d_i=d(x_i,P)$. We have that $B_{d_1}(x_1)\cap B_{d_2}(x_2)=\emptyset$, for otherwise we would have $d(\pi_P(x_1),\pi_P(x_2))\leq 2C$. Let $z_i$ be a point on $\gamma$ such that $d_i-2c\leq d(x_i,z_i)\leq d_i$. Suppose by contradiction that $[z_1,z_2]\cap N_{2C}(P)=\emptyset$. Then $d(\pi_P(z_1),\pi_P(z_2))\leq d(z_1,z_2)/2+C$ by Lemma~\ref{coarsedistdecr:lem}$-(1)$, and in particular $d(z_1,z_2)/2\geq 5C+8c+1$ (notice that $d(\pi_P(z_1),\pi_P(x_i))\leq C$). So,
$$d(x_1,x_2)\leq d(x_1,\pi_P(x_1))+d(\pi_P(x_1),\pi_P(z_1))+d(\pi_P(z_1),\pi_P(z_2))+$$
$$d(\pi_P(z_2),\pi_P(x_2))+d(\pi_P(x_2),x_2)\leq $$
$$(d(x_1,P)+C)+C+(d(z_1,z_2)/2+C)+C+(d(x_2,P)+C)\leq $$
$$d(x_1,z_1)+d(z_1,z_2)+d(z_2,x_2)+5C+4c-d(z_1,z_2)/2\leq $$
$$(d(x_1,x_2)+4c)+5C+4c-d(z_1,z_2)/2<d(x_1,x_2),$$
a contradiction. Therefore $[z_1,z_2]\cap N_{2C}(P)\neq\emptyset$ and in particular $\gamma\cap N_{2C}(P)\neq\emptyset$, as required.
\end{proof}
\end{proof}

\subsection{Main result}

\begin{defn}
\label{altrfree}
A $(1,c)-$quasi-geodesic triangle $\Delta$ is $\calP-$\emph{almost-transverse with constants} $K,D$ if, for each $P\in\calP$ and each side $\gamma$ of $\Delta$, $diam(N_K(P)\cap\gamma)\leq D$.
\par
$\calP$ is \emph{asymptotically transverse-free} if there exist $\lambda,\sigma$ such that for each $D\geq 1$, $K\geq\sigma$ the following holds. If $\Delta$ is a geodesic triangle which is $\calP-$almost-transverse with constants $K,D$, then $\Delta$ is $\lambda D-$thin.
\end{defn}

Recall that a triangle is $\delta-$thin if any point on one of its sides is at distance at most $\delta$ from the union of the other two sides.
\par
The definition of being asymptotically transverse-free only involves geodesic triangles. But, as we will see, if there exists an almost-projection system for $\calP$, then we can deduce something about $(1,c)-$quasi-geodesic triangles as well.

\begin{defn}
$\calP$ is \emph{strongly asymptotically transverse-free} if there exist $\lambda,\sigma$ such that for each $c,D\geq 1$, $K\geq\sigma c$ the following holds. If $\Delta$ is a $(1,c)-$quasi-geodesic triangle which is $\calP-$almost-transverse with constants $K,D$, then $\Delta$ is $\lambda(D+c)-$thin.
\end{defn}

\begin{lemma}
If $\calP$ is asymptotically transverse-free and there exists an almost-projection system for $\calP$, then $\calP$ is strongly asymptotically transverse-free.
\end{lemma}

\begin{proof}
Let $C$ be a projection constant for $\calP$ and let $\lambda_0,\sigma_0$ be the constants such that $\calP$ is asymptotically transverse-free with those constants. We will show that $\calP$ is strongly asymptotically transverse-free for $\sigma=10C+5$. Let $\Delta$ be a $(1,c)-$quasi-geodesic triangle, for $c\geq 1$, which is $\calP-$almost-transverse with constants $K\geq \sigma c ,D\geq 1$, and let $\{\gamma_i\}$ be its sides.
\par
Consider $x,y\in\gamma_i$. We want to prove that any geodesic $\gamma$ from $x$ to $y$ is $\calP-$almost-transverse with ``well-behaved'' constants. Let us start by proving that $d(\pi_P(x),\pi_P(y))\leq D+20C+10c+1$ for each $P\in\calP$. In fact, if that was not the case, by Lemma~\ref{morethanap'2:lem} we would have that $\gamma_i$ intersects $B_{10C+5c}(\pi_P(x))$, $B_{10C+5c}(\pi_P(x))$, so $diam(\gamma_i\cap N_{10C+5c} (P))\geq D+1$ (a contradiction as $\sigma c\geq 10C+5c$). By Corollary~\ref{projneighb:cor} (we can assume $\sigma_0\geq 2C$), we have $diam(\gamma\cap N_{\sigma_0}(P))\leq D+18\sigma_0+82C+10c+1$ for each $P\in\calP$.
\par
By the fact that $\calP$ is asymptotically transverse-free, we obtain that each geodesic triangle whose vertices lie on $\gamma_i$ is $\lambda'-$thin, for $\lambda'=\lambda_0(D+18\sigma_0+82C+10c+1)$. This is all that is needed to apply verbatim the proof of~\cite[Theorem III.H.1.7]{BH-99-Metric-spaces} (which roughly states that in a hyperbolic space quasi-geodesics are at finite Hausdorff distance from geodesics). The constants appearing in the proof are explicitly determined in terms of the hyperbolicity constant $\delta$ ($\lambda'$ plays the role of $\delta$) and the quasi-geodesics constants $\lambda,\epsilon$ (in our case $\lambda=1$, $\epsilon=c$), and one can easily check that the bound on the Hausdorff distance can be chosen to be linear in $\delta+\epsilon$, when fixing $\lambda=1$ (and, say, for $\delta,\epsilon\geq 1$). One can also obtain this remark by a scaling argument.
\par
Hence, each side of $\Delta$ is at Hausdorff distance bounded linearly in $(D+c)$ from the sides of a triangle whose thinness constant is linear in $(D+c)$, so we are done.
\end{proof}

\begin{thm}\label{projasymgrad:thm}
The geodesic metric space $X$ is asymptotically tree-graded with respect to the collection of subsets $\calP$ if and only if $\calP$ is asymptotically transverse-free and there exists an almost-projection system for $\calP$.
\end{thm}
%$\bigcup_{P\in\calP}P$ is $k-$dense in $X$ for some $k\geq 0$,
\begin{proof}
$\Leftarrow:$ Consider an asymptotic cone $Y=C(X,(p_n),(r_n))$ of $X$ and consider the collection $\calP'$ of ultralimits of elements of $\calP$ in $Y$. It is quite clear that elements of $\calP'$ are geodesic, by the assumptions on $\calP$. Also, it is very easy to see that an almost projection system for $\calP$ induces a projection system for $\calP'$.
\par
Let us prove that $\calP'$ is transverse-free. Consider a geodesic triangle $\Delta$ in $Y$. We would like to say that its sides are ultralimits of geodesics in $X$. This is not the case, but, as shown in the following lemma, it is not too far from being true.

\begin{lemma}
Any geodesic $\gamma:[0,l]\to Y$ is the ultralimit of a sequence $(\gamma_n)$ of $(1,c_n)-$quasi-geodesics, where $\mu-\lim c_n/r_n=0$.  
\end{lemma}

\begin{proof}
By \cite[Lemma 9.4]{FLS}, $\gamma$ is a ultralimit of lipschitz paths $\gamma_n$. Let $c_n$ be the least real number so that $\gamma_n$ is a $(1,c_n)-$quasi-geodesic. As the ultralimit of $(\gamma_n)$ is a geodesic, it is readily seen that $\mu-\lim c_n/r_n=0$. 
\end{proof}
Using this lemma, we obtain that $\Delta$, the geodesic triangle we are considering, is the ultralimit of some triangles $\Delta_n$ of $X$ whose sides are $(1,c_n)-$quasi-geodesics and $\mu-\lim c_n/r_n=0$ (as $\Delta$ is $\calP'-$transverse). Suppose that $\Delta$ is $\calP'-$transverse, and let $\lambda,\sigma$ be as in the definition of being strongly asymptotically transverse-free. Let $K_n=\sigma c_n$ and notice that $\Delta_n$ must be $\mu-$a.e. $\calP-$almost-transverse with constants $K_n,D_n$, where $\mu-\lim D_n/r_n=0$. In particular, $\Delta_n$ is $\kappa_n-$thin, where $\kappa_n=\lambda(D_n+c_n)$ so that $\mu-\lim \kappa_n/r_n=0$. This implies that $\Delta$ is a tripod, and hence we showed that $\calP'$ is transverse-free. 
We proved that both conditions of Theorem~\ref{projgrad:thm} are satisfied for $Y$ and $\calP'$, therefore $Y$ is tree-graded with respect to $\calP'$. As $Y$ was any asymptotic cone of $X$, the proof is complete.
\par
$\Rightarrow:$ For each $P\in\calP$, define $\pi_P$ in such a way that for each $x\in X$ we have $d(\pi_P(x),x)\leq d(x,P)+1$. Property $(AP'1)$ is obvious. Property $(AP'2)$ follows directly from Lemma \ref{proj2:lem}$-(1)$.
\par
Let us prove $(AP3)$ (we will use the lemma once again). Let $B$ be a uniform bound on the diameters of $N_H(P)\cap N_H(Q)$ for $P\neq Q\in\calP$ (see Lemma~\ref{coarseinters:lem}), where $H=\max\{tM,L\}$ for $t$ as in Lemma~\ref{qconv:lem}. Fix $P, Q\in\calP$, $P\neq Q$. Suppose that there exist $x,y\in Q$ such that $d(\pi_P(x),\pi_P(y))\geq 2L+B+1$. Consider a geodesic $[x,y]$. It is contained in $N_{tM}(Q)$. Consider points $x',y'$ on $[x,y]$ such that $d(x',\pi_P(x))\leq L$, $d(y',\pi_P(y))\leq L$. Then $d(x',y')\geq d(\pi_P(x),\pi_P(y))-2L\geq B+1$. This is in contradiction with $diam(N_H(P)\cap N_H(Q))\leq B$.
\par
These considerations readily imply $(AP3)$.
\par
We are left to show that $\calP$ is asymptotically transverse-free. Suppose that there is no $\lambda$ such that $\calP$ satisfies the definition of being asymptotically transverse-free with $\sigma=tM$ for $M$ as in Lemma \ref{M:lem} and $t$ as in Lemma \ref{qconv:lem}. Then we have a diverging sequence $(r'_n)$ and geodesic triangles $\Delta_n$ which are $\calP-$almost-transverse with constants $K,D_n$ and optimal thinness constant $r_n=r'_nD_n$. Let $\alpha_n,\beta_n,\gamma_n$ be the sides of $\Delta_n$. We can assume that there exists $p_n\in \alpha_n$ with $d(p_n,\beta_n\cup \gamma_n)= r_n$. Consider $Y=C(X,(p_n),(r_n))$, and let $\alpha,\beta,\gamma$ be the geodesics (or geodesic rays, or geodesic lines) in $Y$ induced by $(\alpha_n),(\beta_n),(\gamma_n)$. Also, let $\calP'$ be the collection of pieces for $Y$ as in the definition of asymptotic tree-gradedness. We claim that for each $P\in \calP'$, $|\alpha\cap P|\leq 1$ (and same for $\beta,\gamma$). This easily leads to a contradiction. In fact, suppose that $\alpha,\beta,\gamma$ all have finite length. Then they form a transverse geodesic triangle that is not a tripod, a contradiction. If at least one of them is infinite, we can reduce to the previous case observing that transverse geodesic rays in $Y$ at finite Hausdorff distance eventually coincide, so that we can cut off parts of $\alpha,\beta,\gamma$ to get once again a transverse geodesic triangle that is not a tripod.
\par
So, suppose that the claim does not hold. Then we can find sequences of points $(x_n),(y_n)$ on $(\alpha_n)$ and a sequence $(P_n)$ of elements of $\calP$ so that $\mu-\lim d(x_n,P_n)/r_n,\mu-\lim d(y_n,P_n)/r_n=0$ but $\mu-\lim d(x_n,y_n)/r_n>0$. By Lemmas \ref{M:lem} and \ref{qconv:lem}, the portion of $\alpha_n$ between $x_n$ and $y_n$ intersects $N_M(P_n)$, so that it contains a subgeodesic in $N_{tM}(P_n)$. It is easily seen that the the length $l_n$ of the maximal such subgeodesic has the property that $\mu-\lim l_n/r_n>0$, in contradiction with $diam(N_{tM}(P_n)\cap \alpha)\leq D_n$.
\end{proof}

\section{Distance formula}
Let $G$ be a relatively hyperbolic group and let $\calP$ be the collection of all left cosets of peripheral subgroups. For $P\in\calP$, let $\pi_P$ be a closest point projection map onto $P$. Denote by $\hat{G}$ the coned-off graph of $G$.
%, that is to say the metric graph obtained from a Cayley graph of $G$ by adding an edge connecting each pair of (distinct) vertices contained in the same left coset of peripheral subgroup. 
Let $\big\{\big\{x\big\}\big\}_L$ denote $x$ if $x>L$, and $0$ otherwise.
We write $A\approx_{\lambda,\mu} B$ if $A/\lambda-\mu\leq B\leq \lambda A+\mu$.
\begin{thm}[Distance formula for relatively hyperbolic groups]\label{distform}
There exists $L_0$ so that for each $L\geq L_0$ there exist $\lambda,\mu$ so that the following holds. If $x,y\in G$ then
\begin{equation}\label{eqn-distform}
	d(x,y)\approx_{\lambda,\mu} \sum_{P\in\calP} \big\{\big\{d(\pi_P(x),\pi_P(y))\big\}\big\}_L+d_{\hat{G}}(x,y).
\end{equation}
\end{thm}

\begin{proof}
Let us start with a preliminary fact. There exists $\sigma$ so that whenever $\gamma_i$, for $i=1,2$, is a geodesic with endpoints in $N_D(P_i)$ for some $P_i\in \calP$ with $P_1\neq P_2$ we have $diam(\gamma_1\cap \gamma_2)\leq \sigma=\sigma(D)$. 
(This is similar to \cite[Lemma 8.10]{Hr-relqconv}, which could also be used for our purposes.) This follows from quasi-convexity of the peripheral subgroups (Lemma \ref{qconv:lem}) combined with the existence of a bound depending only on $\delta$ on the diameter of $N_\delta(P_1)\cap N_\delta(P_2)$ (Lemma \ref{coarseinters:lem}).
%\cite[Theorem 4.1$-(\alpha_1)$]{DSp-05-asymp-cones}. 
So, we have the following estimate for $D_0, M$ as in Lemma \ref{altproj}$-(1)$ for $K=1$ and $C=0$ and $\sigma=\sigma(D_0)$:
\begin{equation}
d(x,y)\geq \sum_{\substack{P\in\calP\\ d(\pi_P(x),\pi_P(y))\geq 2\sigma+2M }} \big(d(\pi_P(x),\pi_P(y))-2\sigma-2M \big). \label{eqn-star}
\end{equation}

Write $A\lesssim_{\lambda,\mu} B$ or $B \gtrsim_{\lambda,\mu} A$ if $A\leq \lambda B+\mu$.
In view of \eqref{eqn-star} and the fact that the inclusion $G\rightarrow \hat{G}$ is Lipschitz 
we have the inequality $\gtrsim_{\lambda,\mu}$ in \eqref{eqn-distform}.
Hence we just need to show that any lift $\tilde{\alpha}$ of a geodesic $\alpha$ in $\hat{G}$ satisfies 
$l(\tilde{\alpha})\lesssim_{\lambda,\mu} R$, where $R$ denotes the right hand side of \eqref{eqn-distform},
with $x,y$ the endpoints of $\tilde{\alpha}$.
Let $\alpha_1,\dots,\alpha_n$ be all maximal subgeodesics of $\tilde{\alpha}$ of length at least some
large $L'$ contained in some left cosets $P_1,\dots,P_n$. We have
$$l(\tilde{\alpha})\approx_{\lambda,\mu} \sum l(\alpha_i)+d_{\hat{G}}(x,y).$$
The endpoints of $\alpha_i$ have uniformly bounded distance from $\pi_{P_i}(x), \pi_{P_i}(y)$ respectively by Lemma \ref{altproj}$-(2)$.
\end{proof}

\subsection{Sample application of the distance formula}\label{samplappl}

We now provide an application of the distance formula. We first need a preliminary lemma. We keep the notation set above.
\begin{prop}
 Let $\phi:G_1\to G_2$ be a $(K,C)-$quasi-isometric embedding between relatively hyperbolic groups so that the image of any left coset of peripheral subgroup of $G_1$ is mapped in the $C-$neighborhood of a left coset of a peripheral subgroup of $G_2$. Then $\phi$ is a $(K',K')-$quasi-isometric embedding at the level of the coned-off graphs, where $K'=K'(K,C)$.
\end{prop}
 
\begin{proof}
 In view of the characterization of projections given in Lemma \ref{altproj}$-(1)$ and the fact that left cosets of peripheral subgroups are coarsely preserved, we see that for each $x\in G_1$ and left coset $P$ of a peripheral subgroup of $G_1$ we have that $\pi_{\phi_\#(P)}(\phi(x))$ is at uniformly bounded distance from $\phi(\pi_P(x))$, where $\phi_\#(P)$ is a left coset of a peripheral subgroup of $G_2$ containing $\phi(P)$ in its $C-$neighborhood.
\par
Fix $x,y$ and let $\hat{\gamma}$ be a geodesic in $\hat{G_1}$ connecting them. Let $\hat{\gamma}_1,\dots, \hat{\gamma}_n$ be the maximal sub-geodesics of $\hat{\gamma}$ that do not contain an edge contained in any left coset of peripheral subgroup $P$ so that $d(\pi_{P}(x),\pi_P(y))$ is larger than some suitable constant $M$. The lift of $\hat{\gamma}_i$ is a quasi-geodesic, and in particular the image $\gamma'_i$ of the lift via $\phi$ is also a quasi-geodesic. The observation we made at the beginning of the proof and the distance formula imply that $\gamma'_i$ is a quasi-geodesic in $\hat{G_2}$ as well. We see then that the image of $\hat{\gamma}$ through $\phi$ is made of a collection of quasi-geodesics of $\hat{G}$ (with uniformly bounded constants) and if $M$ was chosen large enough those quasi-geodesics connect points on a geodesic $\hat{\alpha}$ in $\hat{G}$ from $\phi(x)$ to $\phi(y)$ by Lemma \ref{proj2:lem}. It is not hard to check that $\phi(\hat{\gamma})$ crosses these points in the same order as $\hat{\alpha}$ does, which implies that $\phi(\hat{\gamma})$ is a quasi-geodesic (again, with uniformly bounded constants). In fact, it suffices to show that $\gamma'_i$ does not connect points on opposite sides in $\hat{\alpha}$ of some $\phi_\#(P)$, where $d(\pi_{P}(x),\pi_P(y))> M$. If it did, we would have that the projections of the endpoints of $\gamma'_i$ on $\phi_\#(P)$ are far apart, which implies that the same holds for the endpoints of $\hat{\gamma}_i$, but this is not the case in view of Lemma \ref{altproj}$-(2)$.
\end{proof}

\bibliographystyle{alpha}
\bibliography{bibl}

\end{document}